\documentclass[12pt]{article}

\usepackage{amsthm}

\usepackage[left=1cm,right=1cm,top=0.5cm,bottom=1.5cm,bindingoffset=0cm]{geometry}

\usepackage[T2A]{fontenc}
\usepackage[utf8]{inputenc}
\usepackage[english]{babel}
\usepackage{amsmath}
\usepackage{amsfonts}
\usepackage{mathtools}
\usepackage{graphicx}
\usepackage{systeme}
\usepackage{amssymb}
\usepackage{hyperref}
\usepackage{csquotes}
\usepackage{bbm}
\usepackage[
backend=biber,
style=numeric,
sorting=none
]{biblatex}
\usepackage[normalem]{ulem}
\graphicspath{ {./} }

\newtheorem{theorem}{Theorem}

\newtheorem{lemma}{Lemma}[section]
\newtheorem*{lemma*}{Lemma}
\newtheorem*{den*}{Denotion}

\newtheorem*{remark*}{Remark}

\theoremstyle{remark}
\newtheorem*{remark}{Remark}

\theoremstyle{definition}
\newtheorem{definition}{Definition}[section]

\newcommand{\mP}{\mathbf{P}}
\newcommand{\mE}{\mathbf{E}}

\newcommand{\rk}{\mathrm{rk}}
\newcommand{\Int}{\mathrm{Int}}
\newcommand{\centr}{\mathrm{center}}
\newcommand{\shadow}[2]{\partial^{(#1)}(#2)}
\newcommand{\cG}{\mathcal{G}}
\newcommand{\cGp}{\mathcal{G}^{\prime}}
\newcommand{\cF}{\mathcal{F}}
\newcommand{\cA}{\mathcal{A}}
\newcommand{\cAp}{\mathcal{A}^{\prime}}
\newcommand{\cB}{\mathcal{B}}
\newcommand{\cBp}{\mathcal{B}^{\prime}}
\newcommand{\cBpp}{\mathcal{B}^{\prime\prime}}

\newcommand{\skob}[1]{\left(#1\right)}
\newcommand{\fig}[1]{\left\{#1\right\}}

\title{Probability threshold for the stability of $\alpha(G(n,r,s))$}
\date{}
\author{M. M. Koshelev, A. M. Raigorodskii}

\begin{document}

\maketitle

\sloppy

\begin{abstract}
    Let $G(n,r,s)$ be a graph $(V,E)$ such that $V = \binom{[n]}{r}$ and $E = \{(u, v)| |u \cap v| = s\}. $In this paper we study the probability threshold for the exact stability of the independence number of $G(n, r, s)$ where $r > 2s+2$. We improve the previous result of Pyaderkin that implied asymptotic stability of the independence number.\\ 
    {\bf Keywords:} independence number, stability theorem, generalized Johnson graphs, random subgraphs.
\end{abstract}

\section{Introduction}

In this paper we study one particular family of distance graphs called $G(n,r,s)$ graphs or generalized Johnson graphs. The vertices of such graphs are all possible $r$-element subsets of a set $[n] = \{1, \dotsc, n\}$. An edge is drawn between two vertices if and only if the corresponding sets intersect in exactly $s$ elements.

Generalized Johnson graphs have many applications in combinatorial geometry, coding theory and other branches of combinatorics (see, for example, \cite{Codes,Codes2,FrW,Spectrum}). A lot of literature studies various properties of these graphs (see, for example, \cite{FrFu,RaiMeta,MSp,Shu1,Shu2}).

One of the most studied characteristics of $G(n, r, s)$ is its independence number, $\alpha(G(n,r,s))$. Recall that the independence number of a graph is the biggest $k$ such that there exists $k$ vertices that are pairwise disconnected. Here we will not dive into the details of all results concerning $\alpha(G(n,r,s))$ but rather suggest the survey \cite{RaiMetaAlphaChrom}. Here we only mention two of these results. The first one is the famous Erd\H os--Ko--Rado theorem.

\begin{theorem}[\cite{EKR}]
    \label{EKRth}
    Let $r$ and $n$ be integers such that $n \geq 2r$. Then $\alpha(G(n, r, 0)) = \binom{n-1}{r-1}$.
\end{theorem}

Here we want to highlight a huge role of G.O.H. Katona in the studies of Erd\H os--Ko--Rado theorem. Among them are papers \cite{Katona1964Intersection, Katona2017CompleteIntersection, Katona2017GeneralTwoPartEKR, KatonaWang2026NearlyEKR} and of course a ``proof from The Book'' of EKR theorem obtained in \cite{Katona1972}. And even more than that, G.O.H. Katona was a scientific supervisor of P. Frankl, who also played a big role in the studies of EKR theorem and its modifications (see, for example, \cite{Frankl1978IntersectingFamilies, Frankl1987MaximalDegree, FranklFuredi1991BeyondEKR, FranklTokushige2016Invitation, FranklKupavskii2017SizeSensitive, FranklKupavskii2018EKRVectors, FranklKupavskii2021AlmostIntersecting, FranklKupavskii2021Diversity, FranklKupavskii2023LargeCoveringNumber, FranklKupavskii2024IntersectionVectors}).

The second result concerning the independence number of Johnson graphs we want to mention here is due to Frankl and F\"uredi.

\begin{theorem}[\cite{FrFu}]\label{alpha}
    For any fixed $r$ and $s$ such that $r > 2s + 1$ there exists $n_0(r,s)$ such that for every $n > n_0(r, s)$ we have 
    $$
        \alpha(G(n, r, s)) = \binom{n-s-1}{r-s-1}.
    $$
\end{theorem}

Note that the inequality $\alpha(G(n, r, s)) \geq \binom{n-s-1}{r-s-1}$ is trivial since the set $V_S = \{v \in G(n,r,s)\mid S \subset v\}$, where $S$ is an arbitrary set with $|S| = s + 1$, is independent for any $n$, $r$ and $s$. We can thus see that both Theorem \ref{EKRth} and Theorem \ref{alpha} assert that this set is indeed a maximal independent set in $G(n, r, s)$.

One of the natural questions about $\alpha(G(n,r,s))$ is the stability of this number. To formalize this we need to introduce the following definition.

\begin{definition}
    Let $n$, $r$ and $s$ be integers and fix $p \in [0, 1]$. Then 
    $
        G_p(n, r, s)
    $
    is a random graph obtained from $G(n, r, s)$ by deleting each of its edges independently of the others with probability $1 - p$.
\end{definition}

In view of this definition the question about the stability of $\alpha(G(n,r,s))$ can be reformulated as follows. For which sequences $r(n)$, $s(n)$ and $p(n)$ the equality $\alpha(G_p(n,r,s)) = \alpha(G(n,r,s))$ holds with high probability (that is with probability tending to 1)?

Despite a lot of attention in the past 10 years this question remains widely open. The only case that is almost completely settled is the case $s = 0$. In the series of papers \cite{BogolyubskiiGusevPyaderkinRaigorodskii2014, BogolyubskiiGusevPyaderkinRaigorodskii2015, BNR, BBN, Das}, which culminated in the work \cite{BKL} the following theorem was proved.

\begin{theorem}[\cite{BNR, BBN, Das, BKL}]
    Let $\varepsilon > 0$ be an arbitrary number, and let $n \geq 2r+1$. Then for every $p \geq (1 + \varepsilon)p_0$, where 
    $$
    p_0 = \begin{cases}
        3/4, & n = 2r + 1, \\
        \log\left(n\binom{n-1}{r}\right) / \binom{n-r-1}{r-1},& n > 2r+1,
    \end{cases}
    $$
    the graph $G_p(n,r,0)$ with high probability (w.h.p.) has independence number equal to $\binom{n-1}{r-1}$.
    
    Moreover the equality $\alpha(G_p(n,r,0)) = \binom{n-1}{r-1}$ w.h.p. does not hold for any $p \leq (1 - \varepsilon)p_0$.  
\end{theorem}

The other well-studied case is the case of fixed $r$ and $s$ with $r \leq 2s+1$. Here it can be proved that the independence number is not stable even for $p = 1/2$. For more details we refer to the paper \cite{BogolyubskiiGusevPyaderkinRaigorodskii2015} and the survey \cite{RaiMetaAlphaChrom}.

The case of $r > 2s + 1$ turned out to be trickier. We are aware of two major results here. The first one obtained by Pyaderkin in \cite{Pyaderkin2019MathNotes} finds a coarse probability threshold for a less restrictive type of stability.

\begin{theorem}[\cite{Pyaderkin2019MathNotes}]\label{weakSt}
    Fix $r$ and $s$ such that $r > 2s + 1$. 
    Let
    $$
    p_0(n,r,s)=\frac{\ln n}{n^{r-2s-1}}.
    $$
 Then, as $n\to\infty$,
$$
\mathbf{P}\Bigl[\alpha\bigl(G_p(n,r,s)\bigr)
=\Theta\bigl(\alpha(G(n,r,s))\bigr)\Bigr]
\longrightarrow
\begin{cases}
1, & \text{if } p/p_0\ge 1,\\[1ex]
0, & \text{if } p/p_0\to 0.
\end{cases}
$$
\end{theorem}

The exact stability of $G(n,r,s)$ was also studied by Pyaderkin. In \cite{PyadBug} he claimed the following result.

\begin{theorem}[\cite{PyadBug, PyaderkinFix}]\label{Gnr1}
    Fix $r \geq 5$. Then w.h.p. $\alpha(G_{1/2}(n,r,1)) = \alpha(G(n,r,1))$.
\end{theorem}

In fact even a stronger result (for a wider range of $r$) was claimed, but unfortunately the paper contained an error in one of the lemmas. This error was corrected for fixed $r \geq 5$ in \cite{PyaderkinFix}.

Finally, the same problem was studied in \cite{Pyaderkin2015EKR, OgRai, KarOgRai, RaiKar2022, PokhachevskiyRaigorodskii2026} for a similar family of $G(n,r,<s)$ graphs, where an edge is drawn between subsets that share less than $s$ elements. 

In this paper we significantly generalize Theorem \ref{Gnr1} by widening both the range of applicable $s$ and $p$. We prove the following theorem.

\begin{theorem}\label{main}
    Fix $s \geq 1$, $r > 2s+2$ and $\varepsilon \in (0, 1)$. Let $p_0(n)$, $p_1(n)$ be functions such that $$p_0(n) \leq (1-\varepsilon)\frac{\ln\left(\binom{n}{s+1}\binom{n-s-1}{r}\right)}{\binom{r}{s}\binom{n-r-s-1}{r-2s-1}},$$ 
    $$p_1(n) = \omega\left(\frac{\ln\left(\binom{n}{s+1}\binom{n-s-1}{r}\right)}{\binom{r}{s}\binom{n-r-s-1}{r-2s-1}}\right) = \omega\left(\frac{\ln n}{n^{r-2s-1}}\right).$$
    W.h.p. we have $\alpha(G_{p_0}(n,r,s)) > \alpha(G(n,r,s))$ and $\alpha(G_{p_1}(n,r,s)) = \alpha(G(n,r,s))$.
\end{theorem}

\begin{remark}
    Note that Theorem \ref{main} gives us the coarse probability threshold for the property $\alpha(G_p(n,r,s)) = \alpha(G(n,r,s))$. Moreover, it is likely that the value 
    $$
    \frac{\ln\left(\binom{n}{s+1}\binom{n-s-1}{r}\right)}{\binom{r}{s}\binom{n-r-s-1}{r-2s-1}}
    $$
    is an exact probability threshold for this property. 
\end{remark}

\begin{remark}
    Another thing to note here is that we need a slightly stronger assumption on $r$ in comparison with Theorem \ref{weakSt}. While we believe that Theorem \ref{main} holds for $r = 2s+2$ as well, this case seems to be out of reach for our technique.
\end{remark}

\section{Preliminaries}

Let us introduce some auxiliary definitions and results. The key part of the proof is a structural theorem due to Jiang and Longbrake (see \cite{JiangLongbrake2026}). Before we state it we need to introduce a bunch of definitions.

\begin{definition}
    Let $\cA$ be a set of vertices of $G(n, r, s)$ such that there exists a partition of $[n] = \{1, \dotsc, n\}$ into $r$ subsets $[n] = X_1 \sqcup \dotsc \sqcup X_r$ with the following property: for each $v \in \cA$ and each $i \leq r$ we have $|v \cap X_i| = 1$. We then say that $\cA$ is $r$-partite with parts $X_1, \dotsc, X_r$. 
\end{definition}

\begin{definition}
    Let $\cA$ be a set of vertices of $G(n, r, s)$. Consider a vertex $v \in \cA$. The set $\cA|_v = \{u \cap v \mid u \in \cA \setminus \{v\}\}$ is called the trace of $v$ in $\cA$.
\end{definition}

\begin{definition}
    Let $\cA$ be an $r$-partite set of vertices of $G(n, r, s)$ with parts $X_1, \dotsc, X_r$. Define projection $\pi_{\cA} : 2^{[n]} \to 2^{[r]}$ as follows.
    $$
        \pi_{\cA}(Y) = \{i \in [r] \mid Y \cap X_i \neq \varnothing\}.
    $$
\end{definition}

\begin{definition}
    Let $\cA$ be an $r$-partite set of vertices of $G(n, r, s)$ with parts $X_1, \dotsc, X_r$. The set 
    $$
        \textrm{Int}(\cA) = \{\pi_{\cA}(u \cap v) \mid u, v \in \cA, u \neq v\}
    $$
    is called the intersection structure of $\cA$.
\end{definition}

\begin{definition}
    Let $\cA$ be an $r$-partite set of vertices of $G(n, r, s)$ with parts $X_1, \dotsc, X_r$. Consider a vertex $v$ from $\cA$ and a set $J \subset [r]$. We call $v_J := v \cap \left(\bigcup_{j \in J} X_j\right)$ the narrowing of $v$ on $J$.
\end{definition}

We also need the following notation:

$$
    A(X) = \{v \setminus X \mid v \in A, X \subseteq v\}.
$$

We are finally ready to formulate the theorem of Jiang and Longbrake.

\begin{theorem}[\cite{JiangLongbrake2026}]
    \label{struct}
    Let $\cA$ be a set of vertices of $G(n,r,s)$. Fix an integer $a$ such that $a \geq 2r$. There exists $c = c(r, a) > 0$ (note that $c$ does not depend on $\cA$) and $\cA^{\prime} \subset \cA$ such that
    \begin{enumerate}
        \item $\cA^{\prime}$ is $r$-partite;

        \item $|
        \cA^{\prime}| \geq c(r, a) |\cA|$;

        \item For each $J \in \textrm{Int}(\cA^{\prime})$ and each $v \in \cA^{\prime}$ we have
        $$
            |\cA^{\prime}(v_J)| \geq \max\left\{a, \frac{|\cA^{\prime}|}{2rn^{|J|}}\right\}.
        $$

        \item For each $J \in \textrm{Int}(\cA^{\prime})$ and each $v \in \cA^{\prime}$ $C = \cA^{\prime}(v_J)$ is $a$-diverse (that is, $|C(\{i\})| \leq \frac{1}{a}|C|$ for each $i \in [n]$).

        \item $\textrm{Int}(\cA^{\prime})$ is closed under intersections.
    \end{enumerate}
\end{theorem}


The second theorem we need is the famous Kruskal--Katona theorem (see \cite{Kruskal1963, Katona1968}). Let us start with the definition of the $i$-shadow.

\begin{definition}
    Let $\mathcal{F}$ be a family of sets. Define its $i$-shadow as the family of all $i$-element sets $U$ such that there is at least one set $F \in \mathcal{F}$ that is a superset of $U$. We denote the $i$-shadow of $\mathcal{F}$ as $\partial^{(i)}(\mathcal{F})$.
\end{definition}

Now we are able to formulate the theorem.

\begin{theorem}[\cite{Kruskal1963,Katona1968}]
    \label{KrKat}
    Let $\mathcal{F}$ be a family of $r$-element subsets. Let also $|\mathcal{F}| = \binom{x}{r} := \frac{x(x-1)\dotsc(x-r+1)}{r!}$, where $x$ may be non-integer. Then for each $s < r$ we have
    $$
        |\partial^{(s)}(\mathcal{F})| \geq \binom{x}{s}.
    $$
\end{theorem}

The next theorem and lemma we need were proved in \cite{FrFu}.

The first theorem gives a lower bound on the size of $r-s-1$ shadows in an independent set of $G(n,r,s)$.

\begin{theorem}[\cite{FrFu}, Theorem 5.1]
    \label{shadow_for_ind}
    Let $\cA$ be an independent set in $G(n,r,s)$, $r \geq 2s+1$. Then $|\shadow{r-s-1}{\cA}| \geq c(r)|\cA|$ for some constant $c(r) > 0$.
\end{theorem}

In order to formulate the lemma we need one more auxiliary definition.

\begin{definition}
    Let $\mathcal{M}$ be a family of subsets of $[r]$. Define its rank $\rk(\mathcal{M})$ as the minimum $l$ such that there exists $A \subset [r]$: $|A| = l$ and $A \nsubseteq M$ for any $M \in \mathcal{M} \setminus \{[r]\}$.
\end{definition}

The lemma we need gives us crucial information about the structure of systems with high rank.

\begin{lemma}[\cite{FrFu}, Lemma 5.5]
    \label{centre}
    Let $r$ and $s$ be two integers such that $r > 2s+2$. Let $\mathcal{M}$ be a family of subsets of $[r]$ with $\rk(\mathcal{M}) \geq r-s-1$. Assume that $\mathcal{M}$ is closed under intersection. Then one of the following holds.

    \begin{enumerate}
        \item There exists $M \in \mathcal{M}$ such that $|M| = s$;

        \item $\mathcal{M}$ consists of a set $M$ such that $|M| = s + 1$, all of its supersets and some of its subsets. In this case we call such $M$ the center of $\mathcal{M}$. 
    \end{enumerate}
\end{lemma}

The last two lemmas we need are due to Kupavskii and Shubin (see \cite{ShKup}).

The first one gives us the desired lower bound on the rank and allows us to use Lemma \ref{centre}.

\begin{lemma}[\cite{ShKup}, Claim 11]
    \label{rank}
    Let $\cA$ be an $r$-partite set of vertices of $G(n, r, s)$ such that $|\cA| > \binom{n}{r-s-2}$. Then $\rk(\textrm{Int}(\cA)) \geq r - s - 1$.  
\end{lemma}

The second one shows that each set obtained from Theorem \ref{struct} is either independent or has a lot of edges.

\begin{lemma}[\cite{ShKup}, Claim 12]
    \label{KupSh}
    Consider a set $\cA^{\prime}$ of vertices of $G(n,r,s)$. Let this set satisfy the properties from Theorem \ref{struct} with $a = 2r$. Then $\cA^{\prime}$ is either an independent set or
    $$
        e(\cA^{\prime}) \geq \frac{(r+s)|\cA^{\prime}|^2}{4r|\partial^{(s)}(\cA)|}.
    $$
\end{lemma}

\section{Proof}

\begin{proof}[Proof of the easy part.]
    Consider all sets formed by a star and one vertex that does not intersect with the center of the star. We have $\binom{n}{s+1}\binom{n-s-1}{r}$ such sets. Each one of those sets is independent in $G_{p_0}(n,r,s)$ with probability
    $$
        (1-p_0)^{\binom{r}{s}\binom{n-r-s-1}{r-2s-1}} = \exp\left(-\binom{r}{s}\binom{n-r-s-1}{r-2s-1}p_0 + o(1)\right) \geq \left(\binom{n}{s+1}\binom{n-s-1}{r}\right)^{-1+\varepsilon}(1+o(1)).
    $$
    Thus the expected number of such independent sets is at least
    $$
        \left(\binom{n}{s+1}\binom{n-s-1}{r}\right)^{\varepsilon}(1+o(1))
    $$
    and tends to infinity. 

    Now it only remains to prove that 
    $$
        \mE X^2 = (1+o(1))(\mE X)^2
    $$
    with $X$ being the number of independent sets of such form. First notice that 
    $$
        \mE X^2 = \sum_{\cA_1, \cA_2}\mP(\text{both } \cA_1 \text{ and } \cA_2 \text{ are independent in } G_{p_0}(n,r,s)),
    $$
    where we sum over all pairs of sets of the mentioned form. Let us denote the probability inside this sum by $\mP(\cA_1, \cA_2)$.
    Now rewrite this sum in the following form
    $$
        \mE X^2 = \sum_{S_1, S_2: |S_i| = s+1}\sum_{u_1, u_2:|u_i \cap S_i| = 0}\mP(I(S_1, u_1), I(S_2, u_2)),
    $$
    where $I(S, v)$ is the set formed by the star with center $S$ and vertex $v$. 
    
    Let us now notice that if $u_1 \neq u_2$ the induced subgraphs $I(S_i, u_i)$ share no more than one edge. Indeed, in both of these graphs the only existing edges are edges of the form $(v, u_i)$, where $v$ belongs to the star with center $S_i$. Thus if two such edges coincide it must have both $u_1$ and $u_2$ as its endpoints and there is at most one such edge. Thus for any $u_1 \neq u_2$ and $S_1, S_2, |S_i| = s+1$ we have
    $$
        \mP(I(S_1, u_1), I(S_2, u_2)) \leq (1-p_0)^{2\binom{r}{s}\binom{n-r-s-1}{r - 2s - 1} - 1} = (1+o(1))(1-p_0)^{2\binom{r}{s}\binom{n-r-s-1}{r - 2s - 1}},
    $$
    where $o(1)$ is uniform over all possible choices of $S_i$ and $u_i$.

    Let us now consider a pair $S_1, S_2, |S_i| = s+1$ such that $|S_1 \cap S_2| =: i \neq s+1$, or equivalently, $S_1 \neq S_2$.
    We can easily bound $\mP(I(S_1, u), I(S_2, u))$ for any $u$ not intersecting $S_1 \cup S_2$.
    \begin{multline*}
        \mP(I(S_1, u), I(S_2, u)) = (1-p_0)^{2\binom{r}{s}\binom{n-r-s-1}{r - 2s - 1} - \binom{r}{s}\binom{n-r-2(s+1)+i}{r-s-2(s+1)+i}} \leq \\ \leq (1-p_0)^{2\binom{r}{s}\binom{n-r-s-1}{r - 2s - 1} - \binom{r}{s}\binom{n-r-s-2}{r-2s-2}} = \\ = (1-p_0)^{2\binom{r}{s}\binom{n-r-s-1}{r - 2s - 1}}\exp\skob{(p_0+p_0^2)\binom{r}{s}\binom{n-r-s-2}{r-2s-2}} = \\ =(1+o(1))(1-p_0)^{2\binom{r}{s}\binom{n-r-s-1}{r - 2s - 1}},
    \end{multline*}
    where $o(1)$ is uniform over all possible choices of $S_i$ and $u$.

    We can thus tell that
    \begin{multline*}
        \mE X^2 \leq \sum_{S_1, S_2: |S_i| = s+1}\sum_{u_1, u_2: |u_i \cap S_i| = 0} (1+o(1))(1-p_0)^{2\binom{r}{s}\binom{n-r-s-1}{r - 2s - 1}} + \sum_{S: |S| = s+1}\sum_{u: |u \cap S| = 0}\mP(I(S, u), I(S, u)) = \\ = (1+o(1))(\mE X)^2 + \mE X
    \end{multline*}
    and we are done since $\mE X \to \infty \Rightarrow \mE X = o((\mE X)^2)$.
\end{proof}

\begin{proof}[Proof of the hard part.]
    This proof goes along the lines of Kupavskii and Shubin paper \cite{ShKup} with slight modifications.

    In the proof we use the following notation $z := \frac{p_1n^{r-2s-1}}{\ln n}$ and $d(\mathcal{F}) := \min\limits_{S: |S| = s+1} |\mathcal{F}| - |\cF[S]|$. Note that $z = \omega(1)$. 

    Let us take some set of vertices $\mathcal{F}$ such that $|\mathcal{F}| = \binom{n-s-1}{r-s-1} + 1$. Put $k := d(\mathcal{F})$. Fix $f = \min(z^{\frac{1}{5(s+2)}},\ln n)$.

    Before we start, let us note that it suffices to prove that each such $\cF$ contains $\omega\skob{\frac{k\ln n}{p_1}} = \omega\skob{\frac{kn^{r-2s-1}}{z}}$ edges, since in this case the probability of at least one $\cF$ being independent can be estimated by a trivial union bound. Indeed,
    \begin{multline*}
        \mP\skob{\exists \cF: |\cF| = \binom{n-s-1}{r-s-1} + 1, \cF \text{ is independent in } G_{p_1}(n,r,s)} \leq \\ \leq  \sum_{k = 1}^{\binom{n-s-1}{r-s-1}}\binom{n}{s+1}\binom{\binom{n-s-1}{r-s-1}}{k-1}\binom{\binom{n}{r}}{k}(1-p_1)^{\omega\skob{\frac{k\ln n}{p_1}}} = \\ = \sum_{k = 1}^{\binom{n-s-1}{r-s-1}}\binom{n}{s+1}\binom{\binom{n-s-1}{r-s-1}}{k-1}\binom{\binom{n}{r}}{k}\exp\{-\omega\skob{k\ln n}\} \leq \\ \leq \sum_{k = 1}^{\binom{n-s-1}{r-s-1}}\exp\skob{O(k\ln n)-\omega(k\ln n)} = o(1).
    \end{multline*}

    We start by repeatedly using Theorem \ref{struct} to decompose (most of) $\mathcal{F}$ into $r$-partite subsets. That is, 
    $$
    \mathcal{F} = \mathcal{F}_1 \sqcup \mathcal{F}_2 \sqcup \dotsc \sqcup \mathcal{F}_m \sqcup \mathcal{F}_{m+1},
    $$
    where $\mathcal{F}_i$, $i \leq m$ all satisfy the properties from Theorem \ref{struct}, $|\mathcal{F}_{i}| > \frac{c(r)}{f(n, r, s, p_1)}\binom{n}{r-s-1}$ and $|\mathcal{F}_{m+1}| \leq \frac{1}{f(n, r, s, p_1)}\binom{n}{r-s-1}$. Also note that these bounds imply that $m = O(f(n,r,s,p_1))$.

    Suppose there exists at least one $\mathcal{F}_i$, $i \leq m$ with an edge inside. Then the number of edges inside $\mathcal{F}$ is at least
    $$
        \frac{(r+s)\min_{i \leq m}|\mathcal{F}_i|^2}{4r\max_{i \leq m} |\partial^{(s)}(\mathcal{F}_i)|} \geq \frac{c(r,s)\binom{n}{r-s-1}^2}{f^2(n,r,s,p_1)\binom{n}{s}} \geq \Tilde{c}(r,s)\frac{n^{2r-3s-2}}{f^2(n, r, s, p_1)} = \omega\skob{\frac{kn^{r-2s-1}}{z}}
    $$
    and we are done.

    We thus may assume that all $\mathcal{F}_i$ are independent.

    We can now notice that since $|\mathcal{F}_i| > \binom{n}{r-s-2}$, we can apply Lemma \ref{rank} to show that $\rk(\Int(\mathcal{F}_i)) \geq r-s-1$. By virtue of Lemma \ref{centre} we can say that either there is a set $A \in \Int(\mathcal{F}_i)$ such that $|A| = s$ or $\Int(\mathcal{F}_i)$ has a center. Assume the former holds. This means that there exist two vertices from $\mathcal{F}_i$ such that they share exactly $s$ elements. This contradicts the assumption that each $\mathcal{F}_i$ is independent. 
    
    We can thus assume that each of $\mathcal{F}_i$ has a center. This means that if we take the partition $X_{i,1} \sqcup \dotsc \sqcup X_{i, r}$ (recall that $\cF_i$ is $r$-partite) there exists a set $M_i \subset [r]$, $|M_i| = s+1$ such that for each $v \in \cF_i$ and each $j \in M_i$ we have $|X_{i,j} \cap v| = 1$. Let us define $\centr(v)$ as $v|_{M_i}$. Thus $\centr(v)$ is an $(s+1)$-element subset of $[n]$. 

    Now consider all $\centr(v)$ for all $v \in \cF_1 \sqcup \dotsc \sqcup \cF_m$ and enumerate them as $C_1, \dotsc, C_q$. Also define $\mathcal{G}_j$, $j \leq q$ as follows.
    $$
        \mathcal{G}_j = \{v \in \mathcal{F}_1 \sqcup \mathcal{F}_2 \sqcup \dotsc \sqcup \mathcal{F}_m \mid \centr(v) = C_j\}.
    $$

    The next thing to prove is the inequality
    \begin{equation}\label{sum_bound}
        \sum_{j = 1}^{q}|\shadow{s}{\cG_j(C_j)}| \leq m\binom{n}{s}.
    \end{equation}

    Consider the following sets $S_{i,j} = \shadow{s}{(\cG_j \cap \cF_i)(C_j)}$. If we can show that $S_{i,j_2} \cap S_{i,j_2} = \varnothing$ for $j_1 \neq j_2$, we are done, since 
    $$
        \sum_{j = 1}^{q}|\shadow{s}{\cG_i(C_i)}| \leq \sum_{j = 1}^{q}\sum_{i = 1}^{m}|\shadow{s}{(\cG_j \cap \cF_i)(C_j)}| \leq \sum_{i = 1}^{m}\binom{n}{s} = m\binom{n}{s}.
    $$

    To show that $S_{i,j_1} \cap S_{i,j_2} = \varnothing$ for $j_1 \neq j_2$ we only need to check that for any two vertices $u$ and $v$ from $\cF_i$ such that $\centr(u) \neq \centr(v)$ we have $$|(u \setminus \centr(u)) \cap (v \setminus \centr(v))| \leq |(u \cap v) \setminus (\centr(u) \cap \centr(v))| < s.$$ 
    
    Assume that $|(u \cap v) \setminus (\centr(u) \cap \centr(v))| \geq s$. Let us take a set $Y \subset (u \cap v) \setminus (\centr(u) \cap \centr(v))$  such that $|Y| = s - |\centr(u) \cap \centr(v)|$. Since $M_i \cup \pi(Y)$ is a superset of $M_i$ by Lemma \ref{centre} we have that $(M_i \cup \pi(Y)) \in \Int(\cF_i)$. Since $\Int(\cF_i)$ is closed under intersection (property 5 in Theorem \ref{struct}), we get that $Z = (M_i \cup \pi(Y)) \cap \pi(u \cap v) \in \Int(\cF_i)$. It remains to notice that $Z = \pi(Y) \sqcup \pi(\centr(u) \cap \centr(v))$ and thus $|Z| = s$. This means that there exist two vertices in $\cF_i$ that intersect in exactly $s$ elements. Thus $\cF_i$ is not independent which contradicts with our assumption. 

    Let us renumber the sets $\cG_j$ in such a way that $|\shadow{s}{\cG_1(C_1)}| \geq |\shadow{s}{\cG_2(C_2)}| \dotsc$. We want to prove that the total size of the first $h = h(n,r,s,p_1) = \lceil f^{s+1}(n,r,s,p_1)\rceil$ sets $\cG_j$ is almost the same as the total size of $\cF$. More precisely, we prove that
    \begin{equation}\label{lb_on_sum_cG}
        \sum_{j = 1}^{h}|\cG_j| = (1-o(1))\binom{n}{r-s-1}
    \end{equation}
    holds.

    Indeed, let us take $x_1 \geq \dotsc \geq x_q$ such that $|\shadow{s}{\cG_j(C_j)}| = \binom{x_j}{s}$. Note $x_j$ are not necessarily integer. From inequality \ref{sum_bound} we obtain that $|\shadow{s}{\cG_{h+1}}| \leq \frac{m}{h+1}\binom{n}{s}$. We thus have the following bound on $x_{h+1}$.
    \begin{multline*}
        \frac{(x_{h+1}-s+1)^s}{s!}\leq \binom{x_{h+1}}{s} \leq \frac{m}{h+1}\binom{n}{s} \leq \frac{mn^s}{(h+1)s!} \Rightarrow \\ \Rightarrow x_{h+1}-s+1 \leq \skob{\frac{m}{h+1}}^{1/s}n = O\skob{\frac{n}{f(n,r,s,p_1)}}.
    \end{multline*}
    
    We also have the following bound that is trivial from Theorem \ref{KrKat}.
    $$
        |\cG_j(C_j)| \leq \binom{x_j}{r-s-1}.
    $$
    Note that here we use the inequality $r - s - 1 \geq s$.

    The last observation we need is that the function 
    $$
        R(x) = \frac{\binom{x}{r-s-1}}{\binom{x}{s}} = \frac{s!}{(r-s-1)!}(x-s)\dotsc(x-r+s+2)
    $$
    is increasing for $x \geq r-s-2$ and so $R(x_j) \leq R(x_{h+1})$ for $j \geq h+1$. 

    Combining these bounds, we obtain that
    \begin{multline*}
        \sum_{j = h+1}^{q}|\cG_j| = \sum_{j = h+1}^{q}|\cG_j(C_j)| \leq \sum_{j = h+1}^{q}\binom{x_j}{r-s-1} = \sum_{j = h+1}^{q}\frac{\binom{x_j}{r-s-1}}{\binom{x_j}{s}}\binom{x_j}{s} = \\ = \sum_{j = h+1}^{q}\binom{x_j}{s}R(x_j) \leq R(x_{h+1})\sum_{j = h+1}^{q}\binom{x_j}{s} = R(x_{h+1})\sum_{j = h+1}^{q}|\shadow{s}{\cG_j(C_j)}| \leq m\binom{n}{s}R(x_{h+1}) \leq \\ \leq \frac{s!m}{(r-s-1)!}\binom{n}{s}x_{h+1}^{r-2s-1} = O\skob{\frac{n^{r-s-1}}{f^{r-2s-2}}} = o(n^{r-s-1}).
    \end{multline*}
    
    It remains to notice that $$\sum_{j = 1}^{q}|\cG_j| = |\cF_1 \sqcup \dotsc \sqcup \cF_{m}| = (1-o(1))\binom{n}{r-s-1}$$ and thus 
    $$
    \sum_{j = 1}^{h}|\cG_j| = (1-o(1))\binom{n}{r-s-1} - \sum_{j = h+1}^{q}|\cG_j| = (1-o(1))\binom{n}{r-s-1}.
    $$

    Our next step is to find subsets $\cG^{\prime}_j$ of $\cG_j$, $j \leq h$ that have some good properties and at the same time are almost of the same size as $\cG_j$.

    Define $W = \bigcup_{j = 1}^{h} C_j$, $\tau = \min(z^{1/5},\ln^{s+2} n)$. We want to obtain $\cG_1^{\prime}, \dotsc, \cG_{h}^{\prime}$ that satisfy the following properties.

    \begin{enumerate}
        \item $\cG^{\prime}_j \subseteq \cG_j$ for each $j \in [h]$.
        
        \item 
        $$
        \sum_{j = 1}^{h} |\cG_j| - \sum_{j = 1}^{h} |\cG^{\prime}_j| = o(n^{r-s-1}).
        $$

        \item For any $j \in [h]$ and $G \in \cG^{\prime}_j$, $G \cap W = C_j$.

        \item For any $j \in [h]$ and $X 
        \in \shadow{s}{\cGp_j(C_j)}$, $|\cGp_j(C_j \cup X)| \geq \frac{n^{r-2s-1}}{\tau}$.

        \item For any $j \in [h]$ and $X 
        \in [n] \setminus W$, $|X| \leq s$, the family $\cGp_j(C_j \cup X)$ is $\tau$-diverse.
    \end{enumerate}

    To do so we first fulfill the third condition by removing every $v \in \cG_j$ such that $v \cap W \neq C_j$. Since every such $v$ has at least $s+2$ vertices in $W$, we get that there are no more than 
    $$
        \binom{|W|}{s+2}\binom{n}{r-s-2} \leq |W|^{s+2}n^{r-s-2} \leq r^{s+2}h^{s+2}n^{r-s-2} = o(n^{r-s-1})
    $$
    removed vertices.

    To fulfill conditions 4 and 5 we perform the following algorithm. As long as there exists a pair $\cGp_j$ and $X \in [n]\setminus W$, $|X| \leq s$ violating condition 4 or 5, we remove all $v \in \cGp_j$ that contain $X$ as a subset.

    Note that if the pair $(\cGp_j, X)$ violates condition 4, we remove at most $\frac{n^{r-2s-1}}{\tau}$ vertices at this step. If this pair violates condition 5, we can see that there exists some $i \in [n] \setminus (C_j \cup X)$ such that $|\cGp_j(C_j \cup X \cup \{i\})| \geq \frac{1}{\tau}|\cGp_j(C_j \cup X)|$. Since $|\cGp_j(C_j \cup X \cup \{i\})| \leq \binom{n}{r-s-2-|X|}$, we get that at most
    $
        \tau\binom{n}{r-s-2-|X|}
    $
    vertices are removed at this step. As we can only perform the removal at most once for each pair $(\cGp_j, X)$, the total number of removed vertices is at most
    $$
    h\tau\sum_{x = 0}^{s}\binom{n}{x}\binom{n}{r-s-2-x} + \binom{n}{s}\frac{hn^{r-2s-1}}{\tau} \leq sh\tau n^{r-s-2} + \frac{hn^{r-s-1}}{\tau} = o(n^{r-s-1}).
    $$

    Thus after the end of this procedure the sets $\cGp_j$ satisfy all conditions 1--5.

    Let us now assume that $$\sum_{j = 1}^{h}|\shadow{s}{\cGp_j}| \geq \binom{n}{s}+h^22^s\binom{n}{s-1}+\frac{k}{n^{r-2s-1}\tau}.$$ Notice that for any $v \in \cGp_i$ and for any $X \subset [n]\setminus W$, $|X| = s - |C_i \cap C_j| =: x$ such that $X \subset v$ we have at least $|\cGp_j(X \cup C_j)| - \sum_{a \in G \setminus (X \cup C_j)} |\cGp_j(X \cup C_j \cup \{a\})|$ edges between $v$ and vertices from $\cGp_j$ such that the endpoints of these edges have common set precisely $X \cup (C_i \cap C_j)$. Using property 5 of $\cGp_j$ we can further bound the number of such edges by
    $$
    |\cGp_j(X \cup C_j)| - \sum_{a \in G \setminus (X \cup C_j)} |\cGp_j(X \cup C_j \cup \{a\})| \geq |\cGp_j(X \cup C_j)| - \frac{r}{\tau}|\cGp_j(X \cup C_j)| \geq \frac{\tau - r}{\tau}|\cGp_j(X \cup C_j)|
    $$

    Summing these inequalities over all pairs of $v \in \cGp_i$ and $X \subset [n]\setminus W$, $|X| = v$, we get that the number of edges inside $\cGp_i \cup \cGp_j$ is at least

    \begin{equation*}
        e(\cGp_i \cup \cGp_j) \geq \frac{\tau - r}{\tau}\sum_{X \subset [n]\setminus W,\\ |X| = x}|\cGp_i(C_i \cup X)||\cGp_j(C_j \cup X)|,
    \end{equation*}
    so the total number of edges inside $\mathcal{F} \supseteq (\cGp_1 \cup \dotsc \cup \cGp_h)$ is at least
    \begin{equation}
        \label{ineq1}
        e(\cF) \geq \frac{\tau - r}{\tau}\sum_{i < j}\sum_{X \subset [n]\setminus W,\\ |X| = x}|\cGp_i(C_i \cup X)||\cGp_j(C_j \cup X)|
    \end{equation}

    At the same time we have the following trivial inequality
    \begin{equation}
        \label{ineq2}
        \sum_{X \subset [n]\setminus W, |X| = x}|\cGp_i(C_i \cup X)||\cGp_j(C_j \cup X)| \geq \frac{1}{\binom{r}{s}}\sum_{X \subset [n]\setminus W, |X| = s}|\cGp_i(C_i \cup X)||\cGp_j(C_j \cup X)|.
    \end{equation}
    
    Indeed, on the left hand side we count each pair of vertices $u, v$ such that $u \in \cGp_i$, $v \in \cGp_j$ and $|u \cap v| \geq s$ at least once, while on the right hand side we count each of these pairs at most $\binom{|u \cap v|}{s} \leq \binom{r}{s}$ times. Combining inequalities \ref{ineq1} and \ref{ineq2}, we obtain
    $$
    e(\cF) \geq \frac{\tau-r}{\tau\binom{r}{s}}\sum_{i < j} \sum_{X \subset [n]\setminus W,\\ |X| = s} |\cGp_i(C_i \cup X)||\cGp_j(C_j \cup X)|.
    $$
    
    Let us now prove that under the assumption $\sum_{i}^{h}|\shadow{s}{\cGp_i}| \geq \binom{n}{s}+h^22^s\binom{n}{s-1}+\frac{k}{n^{r-2s-1}\tau}$ we have at least $\frac{k}{n^{r-2s-1}\tau}$ triples $(\cGp_i, \cGp_j, X)$, $X \subset [n] \setminus W$, $|X| = s$ such that $X \in \shadow{s}{\cGp_i} \cap \shadow{s}{\cGp_j}$. Indeed, 
    \begin{multline*}
       \sum_{i}^{h}|\shadow{s}{\cGp_i}| \leq \binom{n}{s} + \sum_{1 \leq i < j \leq h}|\shadow{s}{\cGp_i} \cap \shadow{s}{\cGp_j}| = \binom{n}{s} + \sum_{1 \leq i < j \leq h}\sum_{\substack{X \in \shadow{s}{\cGp_i} \cap \shadow{s}{\cGp_j}}}1 \leq \\ \leq \binom{n}{s} + \sum_{1 \leq i < j \leq h}\sum_{l = 1}^{|C_i \cap C_j|}\sum_{\substack{X \in \shadow{s}{\cGp_i} \cap \shadow{s}{\cGp_j},\\ |X \cap W| = l}}1 + \sum_{1 \leq i < j \leq h}\sum_{\substack{X \in \shadow{s}{\cGp_i} \cap \shadow{s}{\cGp_j},\\ |X \cap W| = l}}1 \leq \\ \leq \binom{n}{s} + h^22^s\binom{n}{s-1} + \sum_{1 \leq i < j \leq h}\sum_{\substack{X \in \shadow{s}{\cGp_i} \cap \shadow{s}{\cGp_j},\\ |X \cap W| = l}}1.
    \end{multline*}
    Using our assumption we see that 
    $$
    \sum_{1 \leq i < j \leq h}\sum_{\substack{X \in \shadow{s}{\cGp_i} \cap \shadow{s}{\cGp_j},\\ X \subset [n] \setminus W}}1 \geq \frac{k}{n^{r-2s-1}\tau},
    $$
    which is exactly what we need.

    Since there are at least $\frac{k}{n^{r-2s-1}\tau}$ triples $(\cGp_i, \cGp_j, X)$, $X \subset [n] \setminus W$, $|X| = s$ such that $X \in \shadow{s}{\cGp_i} \cap \shadow{s}{\cGp_j}$, we have that there are at least as many nonzero summands in 
    $$
        \sum_{i < j} \sum_{X \subset [n]\setminus W,\\ |X| = s} |\cGp_i(C_i \cup X)||\cGp_j(C_j \cup X)|.
    $$
    For each of these summands we know that both $|\cGp_i(C_i \cup X)|$ and $|\cGp_j(C_j \cup X)|$ are at least $\frac{n^{r-2s-1}}{\tau}$ due to property 4. This means that in this case we have
    $$
        e(\cF) \geq \frac{\tau-r}{\binom{r}{s}\tau}\frac{k}{n^{r-2s-1}\tau}\skob{\frac{n^{r-2s-1}}{\tau}}^2 = \omega\skob{\frac{kn^{r-2s-1}}{\tau^4}}.
    $$

    Using the fact that $\tau^4 = o\skob{z}$ we get that the number of edges in $\cF$ is $\omega\skob{\frac{kn^{r-2s-1}}{z}}$ and so we are done.
    
    We may thus assume that $\sum_{i}^{h}|\shadow{s}{\cGp_i}| < \binom{n}{s}+h^22^s\binom{n}{s-1}+\frac{k}{n^{r-2s-1}\tau} = (1+o(1))\binom{n}{s}$.

    We are now going to prove that there exists $i \in [h]$ such that $|\cGp_i| = (1+o(1))\binom{n}{r-s-1}$ thus proving that almost all of the vertices from $\cF$ are a part of one star.  

    Once again, let us renumber $\cGp_i$ in such a way that $|\cGp_1| \geq \dotsc \geq |\cGp_h|$. Furthermore let $y_1, \dotsc, y_h$ be real numbers such that $|\cGp_i| = \binom{y_i}{r-s-1}$. Clearly, $y_1 \geq \dotsc \geq y_h$.

    From Theorem \ref{KrKat} we have that
    \begin{equation*}
    \sum_{i = 1}^{h}\binom{y_i}{s} \leq \sum_{i = 1}^{h}|\shadow{s}{\cGp_i}| \leq (1+o(1))\binom{n}{s}.
    \end{equation*}

    Now recall that 
    $$
        \sum_{i = 1}^{h}|\cGp_i| \geq \sum_{i = 1}^{h}|\cG_i| + o(n^{r-s-1}) \geq (1-o(1))\binom{n}{r-s-1},
    $$
    where the last inequality follows trivially from inequality \ref{lb_on_sum_cG}. On the other end, we have
    $$
    \sum_{i = 1}^{h}|\cGp_i| = \sum_{i = 1}^{h}\binom{y_i}{r-s-1} = \sum_{i =1}^{h}R(y_i)\binom{y_i}{s} \leq R(y_1)\sum_{i = 1}^{h}\binom{y_i}{s} \leq (1+o(1))\binom{n}{s}R(y_1).
    $$
    Combining these two bounds we obtain
    $$
        R(y_1) = \frac{\binom{y_1}{r-s-1}}{\binom{y_1}{s}} \geq (1-o(1))\frac{\binom{n}{r-s-1}}{\binom{n}{s}},
    $$
    from which $y_1 = (1-o(1))n$ follows trivially.

    Let us now return to the initial set $\cF$. We know that the intersection of $\cF$ with one of the stars (without loss of generality we assume that this star is $\fig{v \in G(n,r,s)\mid [s+1] \subset v}$) is $(1-o(1))\binom{n}{r-s-1}$. Let us denote the set of vertices from this intersection by $\cA$ and the set of other vertices from the star by $\cAp$. Furthermore let $\cB := \cF \setminus \cA$. We know that $|\cB| = k$ and $|\cA| = \binom{n-s-1}{r-s-1}-k+1$.
    
    Let $\cBp = \{v \in \cB \mid \deg_{\cA}(v) \leq \frac{1}{2}\binom{n-r-s-1}{r-2s-1}\}$. Assume that $|\cBp| \leq \frac12|\cB| = \frac12k$. Then $\cF$ contains at least $\frac{k\binom{n-r-s-1}{r-2s-1}}{4} = \omega\skob{\frac{kn^{r-2s-1}}{z}}$ edges and we are done. 
    
    It remains to consider the case where $|\cBp| > \frac{1}{2}k$. We start by considering all possible $Q \subset [s+1]$ and taking $Q_0$ such that the number of sets from $\cBp$ that intersect $[s+1]$ by exactly $Q_0$ is maximal. Let $\cBp_Q = \cBp[Q_0]$. We clearly have $|\cBp_Q| \geq \frac{1}{2^{s+1}}|\cBp| \geq \frac1{2^{s+2}}k$. 
    
    Let us now prove the inequality

    \begin{equation}
        \label{lst}
        k = |\cB| \geq \frac{|\shadow{s-|Q_0|}{\cBp_{Q_0}(Q_0)}|\binom{n-r-s-1+|Q_0|}{r-2s-1+|Q_0|}}{2\binom{r-s-1}{s-|Q_0|}}.
    \end{equation}

    Let us fix a pair $(v, S)$ such that $S \subset v$, $|S| = s-|Q_0|$, $v \in \cBp_{Q_0}$. We have at least $\binom{n-r-s-1+|Q_0|}{r-2s-1+|Q_0|}$ vertices $u \in \cA \sqcup \cAp$ such that they intersect $v$ by exactly $S \cup Q_0$. Since the number of vertices between $v$ and $\cA$ is less than 
    $$\frac{1}{2}\binom{n-r-s-1}{r-2s-1} \leq \frac{1}{2}\binom{n-r-s-1+|Q_0|}{r-2s-1+|Q_0|},$$
    at least half of these $u$ lie in $\cAp$. Summing over all $S \in \shadow{s-|Q_0|}{\cBp_{Q}(Q_0)}$, we obtain
    $$
    \frac{1}{2}\binom{n-r-s-1+|Q_0|}{r-2s-1+|Q_0|}|\shadow{s-|Q_0|}{\cBp}| \leq \sum_{S \in \shadow{s-|Q_0|}{\cBp_Q(Q_0)}}|\cAp(S)| \leq \binom{r-s-1}{s-|Q_0|}|\cAp| < \binom{r-s-1}{s-|Q_0|}|\cB|
    $$
    Now inequality \ref{lst} follows  immediately.

    Now apply Theorem \ref{struct} to $\cBp_Q(Q_0)$ with $a = 2(r-|Q_0|)$ (note that we are now considering vertices of $G(n-|Q_0|,r-|Q_0|,s-|Q_0|)$ graph). We obtain an $(r-|Q_0|)$-partite set $\cBpp \subset \cBp_Q(Q_0)$ such that $$|\cBpp| \geq c(r-|Q_0|, 2(r-|Q_0|))|\cBp_Q(Q_0)| \geq \frac{c(r,2r)}{2^{s+2}}k.$$

    Consider two options. The first option is that there is at least one edge inside $\cBpp$. Note that each edge inside $\cBpp$ corresponds to an edge inside $\cBp_Q$, as we add the same set of points to each of the vertices. We can then use Lemma \ref{KupSh} and inequality \ref{lst} to obtain that 
    $$
        e(\cBpp) \geq \frac{c|\cBpp|^2}{|\shadow{s-|Q_0|}{\cBpp}|} \geq \frac{c|\cBpp|^2}{|\shadow{s-|Q_0|}{\cBp}|} \geq \frac{c\binom{n-r-s-1}{r-2s-1+|Q_0|}|\cBpp|^2}{2\binom{r-s-1}{s-|Q_0|}|\cB|} \geq c^{\prime}n^{r-2s-1+|Q_0|}k = \omega\skob{\frac{kn^{r-2s-1}}{z}}.
    $$
    It remains to consider the case where $\cBpp$ is independent. In this case we may apply Theorem \ref{shadow_for_ind} and get that 
    \begin{equation}
        \label{1}
        k = |\cB| \leq C_0(r)|\cBpp| \leq C_1(r)|\shadow{r-s-1}{\cBpp}|
    \end{equation}
    for some $C_1(r) > 0$. 
    
    Let $x$ be a real number such that $|\shadow{r-s-1}{\cBpp}| = \binom{x}{r-s-1}$. Since $$|\shadow{r-s-1}{\cBpp}| \leq \binom{r}{r-s-1}|\cBpp| = o(n^{r-s-1}),$$ we obtain that $x = o(n)$. At the same time 
    \begin{equation}
        \label{2}
        |\shadow{s-|Q_0|}{\cBp_Q(Q_0)}| \geq |\shadow{s-|Q_0|}{\cBpp}| = |\shadow{s-|Q_0|}{\shadow{r-s-1}{\cBpp}}| \geq \binom{x}{s-|Q_0|}.
    \end{equation}
    Combining inequalities \ref{1}, \ref{2} and \ref{lst}, we get
    $$
    C_1(r)\binom{x}{r-s-1} \geq k \geq \frac{|\shadow{s-|Q_0|}{\cBp_Q(Q_0)}|\binom{n-r-s-1+|Q_0|}{r-2s-1+|Q_0|}}{2\binom{r-s-1}{s-|Q_0|}} \geq \frac{\binom{x}{s-|Q_0|}\binom{n-r-s-1+|Q_0|}{r-2s-1+|Q_0|}}{2\binom{r-s-1}{s-|Q_0|}}
    $$
    and thus we have
    $$
        x^{r-s-1} \geq C_2x^{s-|Q_0|}n^{r-2s-1+|Q_0|},
    $$
    which contradicts with the fact that $x = o(n)$.
\end{proof}

\section*{Acknowledgements}

We would like to thank Y. K. Shubin who proofread the text and made a lot of useful remarks.

\printbibliography

\end{document}